\documentclass[12pt]{article}
\usepackage{amssymb}
\usepackage{amsfonts}
\usepackage{mitpress}
\usepackage[T1]{fontenc}

\newtheorem{theorem}{Theorem}

\newtheorem{claim}[theorem]{Claim}

\newtheorem{conjecture}[theorem]{Conjecture}
\newtheorem{corollary}[theorem]{Corollary}

\newtheorem{lemma}[theorem]{Lemma}

\newenvironment{proof}[1][Proof]{\noindent\textbf{#1.} }{\ \rule{0.5em}{0.5em}}
\input{tcilatex}

\usepackage{authblk}

\title{Strong chromatic index of sparse graphs}
\author[1]{Micha\l\ D\k{e}bski\thanks{michal.debski87@gmail.com}}
\author[2,3]{Jaros\l aw Grytczuk\thanks{grytczuk@tcs.uj.edu.pl}}
\author[3]{Ma\l gorzata \'{S}leszy\'{n}ska-Nowak\thanks{m.sleszynska@mini.pw.edu.pl}}
\affil[1]{Faculty of Mathematics, Informatics and Mechanics, University of Warsaw, Warsaw, Poland}
\affil[2]{Faculty of Mathematics and Computer Science, Jagiellonian University, Kraków, Poland}
\affil[3]{Faculty of Mathematics and Information Science, Warsaw University of Technology, Warsaw, Poland}

\begin{document}

\maketitle

\begin{abstract}
A coloring of the edges of a graph $G$ is \emph{strong} if each color class
is an induced matching of $G$. The \emph{strong chromatic index} of $G$,
denoted by $\chi _{s}^{\prime }(G)$, is the least number of colors in a
strong edge coloring of $G$. In this note we prove that $\chi _{s}^{\prime
}(G)\leq (4k-1)\Delta (G)-k(2k+1)+1$ for every $k$-degenerate graph $G$.
This confirms the strong version of conjecture stated recently by Chang and
Narayanan \cite{ChangNarayanan}. Our approach allows also to improve the
upper bound from \cite{ChangNarayanan} for chordless graphs. We get that $%
\chi _{s}^{\prime }(G)\leq 4\Delta -3$ for any chordless graph $G$. Both
bounds remain valid for the list version of the strong edge coloring of
these graphs.
\end{abstract}

\section{Introduction}

A \emph{strong coloring} of a graph $G$ is an edge coloring in which every
color class is an \emph{induced matching}, that is, any two vertices that
belong to distinct edges of the same color are not adjacent. The \emph{%
strong chromatic index} of $G$, denoted by $\chi _{s}^{\prime }(G)$, is the
minimum number of colors in a strong edge coloring of $G$. This notion was
introduced by Erd\H{o}s and Ne\v{s}et\v{r}il around 1985, who were perhaps
inspired by a beautiful construction, due to Ruzsa and Szemer\'{e}di \cite%
{RuzsaSzemeredi}, of dense graphs decomposable into large induced matchings.

A basic conjecture formulated by Erd\H{o}s and Ne\v{s}et\v{r}il states that $%
\chi _{s}^{\prime }(G)\leq \frac{5}{4}\Delta ^{2}$ for every graph $G$ with
maximum degree $\Delta $. If true, this bound is sharp, as is seen on the
graph obtained from the cycle $C_{5}$ by blowing-up each vertex to an
independent set of size $\frac{\Delta }{2}$. For the general case, the
trivial upper bound on $\chi _{s}^{\prime }(G)$ is $2\Delta ^{2}-2\Delta +1$%
, as is seen easily by applying the greedy algorithm for any ordering of the
edges. Using the probabilistic method, Molloy and Reed \cite{MoloyReed}
improved the trivial bound to $(2-\epsilon )\Delta ^{2}$ for $\Delta $
sufficiently large, where $\epsilon $ is a small constant around $\frac{1}{50%
}$. They claim that this result can be slightly improved, but increasing $c$
above $\frac{1}{10}$ requires essentially new ideas. Another intriguing
conjecture is due to Faudree, Gy\'{a}rf\'{a}s, Schelp, and Tuza \cite%
{FaudreeEtAl}. It\ states that $\chi _{s}^{\prime }(G)\leq \Delta ^{2}$ when 
$G$ is a bipartite graph of maximum degree $\Delta $. Its stronger version
formulated by Brualdi and Quinn \cite{BrualdiQuinn} states that $\chi
_{s}^{\prime }(G)\leq \Delta _{1}\Delta _{2}$ if $G$ is a bipartite graph
and $\Delta _{i}$ is the maximum degree of a vertex in the $i$-th partition
class.

In this note we focus on graphs with bounded degeneracy. Recall that a graph 
$G$ is $k$\emph{-degenerate} if every subgraph of $G$ contains a vertex of
degree at most $k$. Chang and Narayanan \cite{ChangNarayanan} proved
recently that $\chi _{s}^{\prime }(G)\leq 10\Delta -10$ for a $2$-degenerate
graph $G$. They also conjectured that for any $k$-degenerate graph $G$ we
have a linear bound $\chi _{s}^{\prime }(G)\leq ck^{2}\Delta $, where $c$ is
an absolute constant. We confirm this conjecture in a stronger form, by
proving that inequality 
\[
\chi _{s}^{\prime }(G)\leq (4k-1)\Delta -k(2k+1)+1 
\]%
holds for every $k$-dgenerate graph $G$. Our proof is algorithmic: we
produce appropriate ordering of the edges of a graph $G$, and then apply the
greedy coloring algorithm. By the same technique, using a structural lemma
from \cite{ChangNarayanan}, we get an improved bound in a special case of 
\emph{chordless} graphs (graphs in which every cycle is induced). These
graphs form a proper subclass of $2$-degenerate graphs, and contain graphs
with cycle lengths divisible by four. It was conjectured in \cite%
{FaudreeEtAl} that for such graphs the strong chromatic index is linear in $%
\Delta $. We prove that every chordless graph $G$ of maximum degree $\Delta $
satisfies $\chi _{s}^{\prime }(G)\leq 4\Delta -3$.

\section{The results}

In our proof we will use the following simple lemma concerning local
structure of $k$-degenerate graphs. We say that a vertex $v$ is \emph{nice}
in a graph $G$ if $v$ has at least one neighbor of degree at most $k$, and
at most $k$ neighbors of degree strictly greater that $k$.

\begin{lemma}[{\protect\cite[Lemma 3]{ChangNarayanan}}]
\label{lemma_kdegenerateStructure} Every $k$-degenerate graph $G$ with at
least one edge contains a nice vertex.
\end{lemma}

\begin{proof}
Let $X\subseteq V(G)$ be the set of all vertices of degree at most $k$ in $G$%
. If there is an edge $e$ between any two vertices from $X$, then each of
the end points of $e$ is nice. Otherwise, when $X$ is an independent set,
consider the subgraphs $H$ induced by the set $V(G)\setminus X$ of all the
vertices whose degree in $G$ is strictly greater than $k$. Notice that $%
V(G)\setminus X$ is nonempty because $G$ has at least one edge. Let $v\in
V(G)\setminus X$ be any vertex whose degree in the subgraph $H$ is at most $%
k $ (which must exists, as $H$ is also $k$-dgenerate). We claim that $v$ is
nice in $G$. Indeed, it has at most $k$ neighbors in $V(G)\setminus X$, and
since it does not belong to $X$, it must have at least one neighbor in $X$.
\end{proof}

Our plan for the proof of the main result is simple: we will construct an
ordering of the edges of a $k$-degenerate graph $G$ such that for any edge $%
e $, the number of edges that are within distance one from $e$ and come
after $e$ in the order, is appropriately bounded. The desired coloring is
then obtained by coloring edges greedily in the reverse order. The ordering
will be constructed in a number of steps. Initially all edges are white. In
each step we will grey out a bunch of edges and add them to the ordering (at
the end). During the whole procedure we will keep under control the number
of white edges incident to grey edges, thereby obtaining a desired result.

\begin{theorem}
Every $k$-degenerate graph $G$ of maximum degree $\Delta $ satisfies%
\[
\chi _{s}^{\prime }(G)\leq (4k-1)\Delta -k(2k+1)+1. 
\]
\end{theorem}

\begin{proof}
We will iteratively construct an (ordered) list of edges of $G$, starting
with $L_{0}$ being an empty list. We will find an increasing chain of lists $%
L_{1},L_{2},\ldots $ such that $L_{i}$ is obtained from $L_{i-1}$ by
appending a (nonzero) number of new edges. For some $s$, the last list $%
L_{s} $ will contain all edges of $G$, and it will be the desired ordering
of the set $E(G)$.

Suppose that $L_{i}$, with $i>0$, is defined, and let $X_{i}\subset E(G)$ be
a set of edges that appear in $L_{i}$. We may imagine the edges from $X_{i}$
to be greyed out, while the rest of edges being still white. Let $%
H_{i}=(V(G),E(G)\setminus X_{i})$, and let $v_{i}$ be a nice vertex in $%
H_{i} $. By Lemma \ref{lemma_kdegenerateStructure}, such a vertex must
exist. Now, let $Y_{i}$ be the set of white edges incident with $v_{i}$ and
a vertex of degree at most $k$, that is,%
\[
Y_{i}=\{v_{i}w\in E(H_{i}):\deg _{H_{i}}(w)\leq k\}. 
\]%
Take $L_{i+1}$ to be the list obtained from $L_{i}$ by appending edges from $%
Y_{i}$ in any order. Clearly, $Y_{i}$ is nonempty, so for some $s$, the list 
$L_{s}$ will contain all edges of $G$. In order to prove the desired
property of $L_{s}$ we will keep the following invariant: the number of
white edges incident to each end point of a grey edge is at most $k$. More
formally, we shall prove the following claim.

\begin{claim}
For every $i=0,1,\ldots ,s$, and every vertex $v\in V(G)$, if $v$ is
incident to at least one edge from $X_{i}$, then it is incident to at most $%
k $ edges outside $X_{i}$.
\end{claim}

To prove the claim we use induction on $i$. For $i=0$ we have $%
X_{i}=\emptyset $, so there is nothing to prove. Now, suppose that the
invariant holds for some $i$ (where $0\leq i<s$), and consider some vertex $%
v $ incident with at least one edge from $X_{i+1}$. Note that if $v$ is
incident to an edge from $X_{i}$, then the statement follows from induction
hypothesis, as $X_{i}\subset X_{i+1}$. In the remaining case, $v$ is
incident to an edge from $Y_{i}$. So, either $v\neq v_{i}$ and $\deg
_{H_{i}}(v)\leq k$ by definition of $H_{i}$, or $v=v_{i}$ and, as a nice
vertex in $H_{i}$, it is incident to at most $k$ edges outside $X_{i}\cup
Y_{i}$. This finishes proof of the Claim.

Now, for any edge $e\in E(G)$, we will count the number of edges that are
within distance one from $e$ and appear on $L_{s}$ later than $e$. Let $i$
be the unique number for which $e\in Y_{i}$, and let $e=v_{i}w$, with $v_{i}$
being nice in $H_{i}$. We will estimate the number of white edges (outside $%
X_{i}$) that are within distance one from $e$.

We will say that a vertex $v$ \emph{sees} an edge $e$ if $v$ is incident to $%
e$, or it is a neighbor of a vertex incident to $e$. First we count white
edges that are seen from vertex $v_{i}$ (except $e$ itself). Let $a$ be the
number of grey edges incident to $v_{i}$. By the Claim, there are at most $%
ak $ white edges in total, incident to the end points of such edges (other
than $v_{i}$). Next, let $b$ denote the number of neighbors of $v_{i}$ in $%
H_{i}$ whose white degree is at most $k$. So, we have another $bk$ white
edges seen from $v_{i}$. Finally, let $c$ be the number of neighbors of $%
v_{i}$ in $H_{i}$ whose white degree is strictly larger than $k$. Thus we
have at most $ak+bk+c\Delta $ white edges in total that can be seen from $%
v_{i}$. But $a+b+c\leq \Delta -1$ and $c\leq k$ (since $v_{i}$ is nice in $%
H_{i}$). So, we get that%
\[
ak+bk+c\Delta =(a+b+c)k+c(\Delta -k)\leq (\Delta -1)k+k(\Delta -k). 
\]%
Now, we count similarly the number of white edges seen from the vertex $w$.
Let $p$ and $q$ be the number of white and grey edges (different from $e$)
incident to $w$, respectively. Arguing as before we get at most $p\Delta +qk$
white edges seen from $w$. Since $\deg _{H_{i}}(w)\leq k$, we have $p\leq
k-1 $. Also, $p+q\leq \Delta -1$. Hence%
\[
p\Delta +qk=p(\Delta -k)+(p+q)k\leq (k-1)(\Delta -k)+(\Delta -1)k. 
\]%
Putting these two estimates together, we obtain that the total number of
white edges seen from the end points of $e$ is at most%
\[
(4k-1)\Delta -k(2k+1). 
\]%
To finish the proof of the theorem, notice only that we need one more color
to guarantee success of the greedy coloring procedure with respect to the
linear ordering of the edges we have produced.
\end{proof}

As every chordless graph $G$ is $2$-degenerate, our theorem implies that $%
\chi _{s}^{\prime }(G)\leq 7\Delta -9$. However, one may improve it by using
the following lemma from \cite{ChangNarayanan}.

\begin{lemma}[{\protect\cite[Lemma 7]{ChangNarayanan}}]
\label{lemma_chordlessStructure} Every chordless graph $G$ contains a nice
vertex with at most one neighbor of degree strictly greater than two.
\end{lemma}

Indeed, the lemma implies that in the process of constructing lists $L_{i}$
we can keep the stronger invariant, namely, that the number of white edges
incident to each end point of a grey edge is at most one. By a similar
counting one can obtain the following theorem (a detailed proof is omitted).

\begin{theorem}
Every chordless graph $G$ of maximum degree $\Delta $ satisfies 
\[
\chi _{s}^{\prime }(G)\leq 4\Delta -3. 
\]
\end{theorem}

An important subclass of $k$-dgenerate graphs are graphs of treewidth at
most $k$. For these graphs a much better bound for the strong chromatic
index can be obtained using a different argument. The following result was
(implicitly) proved in \cite{FaudreeEtAl2}, and we repeat its nice proof
here for completeness.

\begin{theorem}
Let $\mathcal{G}$ be a proper minor-closed class of graphs. Let $r=r(%
\mathcal{G})$ be the maximum chromatic number of a member of $\mathcal{G}$.
Then every graph $G\in \mathcal{G}$ of maximum degree $\Delta $ satisfies%
\[
\chi _{s}^{\prime }(G)\leq r(\Delta +1). 
\]
\end{theorem}

\begin{proof}
Suppose we have $\Delta +1$ basic colors, each having $r$ different shades.
Start with the usual proper edge coloring of $G$ using at most $\Delta +1$
basic colors. Let $M$ be a fixed color class (which is a usual matching in $%
G $). Consider a graph $G_{M}$ whose vertex set is $M$, with two edges $%
e,f\in M$ adjacent whenever at least one pair of their corresponding end
points is joined by an edge in $G$. In other words, $G_{M}$ is obtained by
taking a subgraph of $G$ induced by the end points of all edges in $M$, and
then their contraction. Clearly, $G_{M}$ belongs to $\mathcal{G}$, so it has
a proper vertex coloring using at most $r$ colors. It follows that we may
color the edges of $M$ with $r$ shades of the basic color of $M$ so that
each shade forms now an induced matching in $G$. Repeating this operation
for every matching of the initial edge coloring of $G$ we get the desired
strong coloring using at most $r(\Delta +1)$ colors.
\end{proof}

Since graphs of treewidth at most $k$ are $(k+1)$-colorable we get the
following result.

\begin{corollary}
Every graph of treewidth at most $k$ satisfies%
\[
\chi _{s}^{\prime }(G)\leq (k+1)(\Delta +1). 
\]
\end{corollary}

Notice that all upper bounds obtained by greedy colorings remain valid for
the natural list version of the strong chromatic index.

\section{Final remarks}

Our main theorem asserts, that the strong chromatic index of graphs of
bounded degeneracy is at most linear in their maximum degree, which is best
possible up to a multiplicative constant. This covers all natural classes of
graphs that have bounded \emph{average} degree, such as graphs of bounded
treewidth, graphs of bounded arboricity, planar graphs (where natural means
closed under taking subgraphs). On the other hand, when average degree is a
function $d(n)$ of the number of vertices in a graph, our result implies a
bound on $\chi _{s}^{\prime }(G)$ that is not better than $cd(n)^{2}$ for
some constant $c$.

It is natural to wonder whether our bound is optimal in some sense, in terms
of average degree. Namely, does there exist a family of graphs of average
degree $d(n)$, greater than a constant, that have strong chromatic index
bounded by a linear function of $d(n)$? Note that it is in fact a question
concerning a nontrivial lower bound on $\chi _{s}^{\prime }(G)$ that depends
only on the average degree. As we suspect that such a bound may not exist,
we conjecture the following.

\begin{conjecture}
\label{conjecture_bezograniczen} There exists a function $d(n)$ that goes to
infinity such that for any constant $c$ there exists a graph on $n$ vertices
of average degree $d(n)$, that have strong chromatic index at most $cd(n)$.
\end{conjecture}

Surprisingly, a slightly weaker assertion holds even for very dense graphs.
Alon, Moitra, and Sudakov \cite{Alonetal} construct graphs of average degree 
$n-o(n)$ that have strong chromatic index of order $O(n^{1+o(1)})$. This
implies that an analog of Conjecture \ref{conjecture_bezograniczen}, where
we replace bound on strong chromatic index with $d(n)^{1+\epsilon }$, is
true even for $d(n)=n-o(n)$. However, their results do not imply anything
for graphs of strong chromatic index at most $cd(n)$.

Knowing that there are graphs of average degree $d(n)$ and chromatic index
at most $d(n)^{1+\epsilon }$, a natural continuation of our research is to
find a characterization of those graphs that would allow us to strengthen
the upper bound given by Theorem 2 and, in a wider perspective, reveal some
new properties of this intriguing chromatic parameter.


\begin{thebibliography}{9}
\bibitem{Alonetal} N. Alon, A. Moitra, B. Sudakov, Nearly complete graphs
decomposable into large induced matchings and their applications, in Proc.
STOC, 2012, pp.1079--1090.

\bibitem{BrualdiQuinn} R. A. Brualdi, J. J. Quinn Massey, Incidence and
strong adge colorings of graphs. Discrete Math. 122 (1993), 51-58, 1993.

\bibitem{ChangNarayanan} G. J. Chang, N. Narayanan, Strong chromatic index
of 2-degenerate graphs, J. Graph Theory. doi: 10.1002/jgt.21646, 2012.

\bibitem{FaudreeEtAl} R. J. Faudree, A. Gyarfas, R. H. Schelp, Zs. Tuza.
Induced matchings in bipartite graphs, Discrete Math. 78 (1989), 83--87.

\bibitem{FaudreeEtAl2} R. J. Faudree, A. Gyarfas, R. H. Schelp, Zs. Tuza.
The strong chromatic index of graphs, Ars Combinatoria, 29B (1990), 205--211.

\bibitem{Mahdian} M. Mahdian. The strong chromatic index of C4-free graphs.
Random Struct. Algorithms, 17 (2000), 357--375.

\bibitem{MoloyReed} M. Molloy, B. Reed, A bound on the strong chromatic
index of a graph, J. Combin. Theory Ser. B, 69 (1997), 103--109.

\bibitem{RuzsaSzemeredi} I. Ruzsa, E. Szemer\'{e}di, Triple Systems with no
Six Points Carrying Three Triangles. Colloquia Mathematica Societatis J anos
Bolyai, pp. 939--945, 1978.
\end{thebibliography}
\end{document}